\documentclass[10pt, reqno]{amsart}
\calclayout
\usepackage[utf8]{inputenc}
\usepackage[english]{babel}
\usepackage[dvipsnames]{xcolor}
\usepackage{blindtext}
\usepackage{enumitem}
\usepackage{amssymb}
\usepackage{commath}
\usepackage{amsmath}
\usepackage{mathtools}
\usepackage{amsthm}
\usepackage{nicefrac}
\newtheorem{theorem}{Theorem}[section]
\newtheorem{proposition}[theorem]{Proposition}
\newtheorem{lemma}[theorem]{Lemma}
\newtheorem{corollary}[theorem]{Corollary}

\theoremstyle{definition} 

\theoremstyle{remark}
\newtheorem{remark}[theorem]{Remark}

\newcommand{\Z}{\mathbb{Z}}
\usepackage{mathtools}
\usepackage{comment}
\usepackage{amsrefs}
\usepackage{microtype}
\DeclarePairedDelimiter\ceil{\lceil}{\rceil}
\DeclarePairedDelimiter\floor{\lfloor}{\rfloor}
\newcommand{\Mod}[1]{\ (\mathrm{mod}\ #1)}

\usepackage{hyperref}
\hypersetup{colorlinks=True, linkcolor=blue, urlcolor=blue, citecolor=blue}

\title{Explicit Burgess Bound for Composite Moduli}
\author{Niraek Jain-Sharma,  Tanmay Khale, Mengzhen Liu}

\begin{document}

\maketitle
\begin{abstract}
    We prove an explicit version of Burgess' bound on character sums for composite moduli. 
\end{abstract}
\section{Introduction}\label{intro}
A recurrent goal in analytic number theory is that of bounding the character sum
$$S(M,N) := \sum_{n=M+1}^{M+N} \chi(n),$$
where $\chi$ is a non-principal Dirichlet character modulo $q$. By the orthogonality of characters, it is easily seen that $S(M,N) < q$. The first improvement over this bound was made by P\'olya \cite{polya} and Vinogradov \cite{vinogradov} independently in 1918, who showed that for any non-principal Dirichlet character, $S(M,N) \ll \sqrt{q}\log q$. Assuming the Generalized Riemann Hypothesis, this was improved to $S(M,N) \ll \sqrt{q}\log\log q$ by Montgomery and Vaughan~\cite{mv2} in 1977, and since a 1932 result of Paley~\cite{paley} states that $S(M,N) \gg \sqrt{q} \log\log q$ for infinitely many quadratic characters, the Montgomery--Vaughan bound is the best possible one independent of the length of the interval.

However, for many applications, it suffices to bound character sums over \textit{short} intervals, and for these character sums, it is possible to obtain significant improvements over the above square-root estimates by incorporating the interval length. For instance, for intervals of length less than $\sqrt{q}$, the trivial bound $S(M,N) \leq N$ from the triangle inequality is stronger than even the Montgomery--Vaughan bound implied by GRH. In 1962, Burgess demonstrated a landmark upper bound which dramatically improves the estimates given by P\'olya--Vinogradov or the trivial bound for intervals of length $q^{\frac{1}{3} + \varepsilon} \leq N \leq q^{\frac{5}{8}-\varepsilon}$ in all cases, and for intervals of length $q^{\frac{1}{4}+\varepsilon} \leq N \leq q^{\frac{5}{8}-\varepsilon}$ when the conductor of $\chi$ is cube-free. In particular, Burgess~\cite{burgess2}~\cite{burgess3} showed that for all non-principal characters $\chi$ and all positive integers $r$, 
if either the conductor of $\chi$ is cube-free or $r  =2,3$, 
\begin{equation}\label{burgess}
    S(M,N) \ll_{r,\varepsilon} N^{1-\frac{1}{r}} q^{\frac{r+1}{4r^2} + \varepsilon}.
\end{equation}
Burgess' bound has numerous implications in a wide range of number-theoretic contexts. Burgess himself first applied his estimate in~\cite[Theorem 3]{burgess0} to bounding primitive roots: namely, if $g(p)$ is the least primitive root modulo a prime $p$, Burgess proved
\begin{equation*}
    g(p) \ll_\varepsilon p^{\frac{1}{4} + \varepsilon}.
\end{equation*}
Another important application of Burgess' bound is to Vinogradov's problem of bounding the least quadratic non-residue $n_p$ modulo a prime $p$. As a consequence of Burgess' bound,
\begin{equation*}
    n_p \ll_\varepsilon p^{\frac{1}{4\sqrt{e}} + \varepsilon}.
\end{equation*}
Finally, the deepest application of the Burgess bound is to certain estimates for Dirichlet $L$-functions on the critical line, known as \textit{subconvexity} bounds. In this direction, Burgess proved~\cite[Theorem 3]{burgess2}
\begin{equation*}
    L(1/2 +it,\chi) \ll_{t,\varepsilon} q^{\frac{3}{16} + \varepsilon},
\end{equation*}
which was improved by Heath-Brown~\cite{heath-brown} to
\begin{equation}\label{youngimprovement}
        L(1/2 + it,\chi) \ll_{\varepsilon} \left(\left(\abs{t}+2\right)q\right)^{\frac{3}{16} + \varepsilon}.
\end{equation}
As a notable special case of the above, Burgess' bound implies $L(1/2,\chi) \ll q^{\frac{3}{16} + \varepsilon}$.

It is conjectured, and follows from the Generalized Riemann Hypothesis, that  $S(M,N) \ll_\varepsilon \sqrt{N}q^{\varepsilon}$. In the contexts above, GRH implies respectively that $g(p) \ll (\log p)^6$ (shown by Shoup~\cite{shoup}), $n_p \leq \log^2 p$ (shown in inexplicit form by Ankeny~\cite{ankeny}, and with constant $1$ by Lamzouri et al.~\cite{sound}), and $L(1/2+it) \ll ((1+\abs{t})q)^\varepsilon$ (i.e., the Lindel\"of hypothesis for Dirichlet $L$-functions). However, both the Burgess bound and its implications have stubbornly resisted improvement in the decades following its publication. As a testament to this difficulty, we note that in the contexts of primitive roots and quadratic non-residues, Burgess' estimate is still the state-of-the-art, while for Dirichlet $L$-functions the estimate in~\eqref{youngimprovement} has been improved only recently: building on the work of Conrey and Iwaniec~\cite{conrey}, Petrow and Young~\cite{young0, young} proved in a groundbreaking series of papers in 2018 and 2019 that the exponent of $\nicefrac{3}{16}$ in~\eqref{youngimprovement} may be reduced to $\nicefrac{1}{6}$ for all Dirichlet $L$-functions. 

In the prime modulus case, \textit{explicit} Burgess-strength bounds have been worked out by  Booker \cite{booker}, McGown \cite{mcgown}, Trevi\~{n}o \cite{trevino}, and Francis~\cite{francis}. In particular,~\cite[Corollary 1]{trevino}, is a totally explicit version of the Burgess inequality as stated in~\eqref{burgess} for all positive integers $r$ and all prime moduli $p \geq 10^7$. 

However, it recently came to light---owing to a question by Bilu \cite{MO} posted on MathOverflow---that a Burgess-strength bound had never been worked out explicitly for composite moduli. In Theorem~\ref{main_thm3}, following ideas in Burgess \cite{burgess1} and the presentation of Montgomery and Vaughan \cite[Theorem 9.27]{mv}, we provide such a bound for primitive characters.
\subsection*{Main Results}
Let
\begin{align*}
    d(n) &= \abs{\{k \in \Z^+: k \mid n\}}, \\
    \omega(n) &= \abs{\{p \in \Z^+: p \text{ prime, } p \mid n\}}.
\end{align*}
\begin{theorem}\label{main_thm3}
    Let $\chi$ be a primitive character with modulus $q \geq e^{e^{9.594}}$. Then, for $N \leq q^{\frac{5}{8}}$,
    \begin{equation*}
    \abs{S(M,N)} \leq \sqrt{N}q^{\frac{3}{16}} \cdot 9.07 \log^{\frac{1}{4}}(q) (2^{\omega(q)} d(q))^{\frac{3}{4}} (\nicefrac{q}{\varphi(q)})^{\frac{1}{2}}
\end{equation*}
\end{theorem}
By Theorem~\ref{phi_bound}, $\frac{q}{\varphi(q)} \ll \log\log q$ (with an explicit constant). Applying this bound in Theorem~\ref{main_thm3} yields the following corollary: 
\begin{corollary}\label{totally_explicit_thm3}
Let $\chi$ be a primitive character with modulus $q \geq e^{e^{9.594}}$. Then, for $N \leq q^{\frac{5}{8}}$,
 \begin{equation*}
\abs{S(M,N)} \leq \sqrt{N}q^{\frac{3}{16}} \cdot 12.11 \log^{\frac{1}{4}}(q) (2^{\omega(q)} d(q))^{\frac{3}{4}}\left(\log \log q+\frac{1.69}{\log \log q}\right)^{\frac{1}{2}}.
 \end{equation*}
\end{corollary}
We also remark that while we were working on this problem, Bilu's question was independently answered by Bordignon~\cite{bordignon}. Our Corollary~\ref{totally_explicit_thm3} has a factor of $\log^{\frac{1}{4}}(q)$ in place of the $\log^{\frac{1}{2}}(q)$ in~\cite[Theorem 1.1]{bordignon}, and thus improves Bordignon's estimate by a factor of $\log^{\frac{1}{4}}(q)$. 
\subsection*{Organization}
The paper is structured as follows. In Sections~\ref{Bounding_Aq_A} and~\ref{Lemma8}, we prove several lemmas that are used in our proof of Theorem~\ref{main_thm3}: specifically, in Section~\ref{Bounding_Aq_A}, we bound from below the number of positive integers up to $A$ coprime to a given modulus $q$, and in Section 3, compute an explicit fourth-moment bound for character sums. In Section~\ref{1.2_proof}, we utilize Lemmas~\ref{Aq_bound3} and~\ref{lemma:fourth_moment_of_character_sums} to prove Theorem~\ref{main_thm3} along the lines of \cite[Theorem 9.27]{mv}.
\subsection*{Acknowledgements}
The authors would like to thank P\'eter Maga for mentoring this research project, as well as Gergely Harcos for his role in the inception of the project and the editing process. The authors would also like to thank the anonymous referee for their exceptionally careful reading of this paper, which has been greatly improved as a result of their feedback. This research was conducted under the auspices of the Budapest Semesters in Mathematics program.
\section{Bounding $\frac{A_q}{A}$}\label{Bounding_Aq_A}
Let $A,q$ be a positive integers with $1 \leq A \leq q$. Define
\begin{equation*}
    A^q = \{n \in \Z : 1 \leq n \leq A, (n,q) = 1 \},
\end{equation*}
and $A_q = \abs{A^q}$. In this section, we will prove a lower bound for the quantity $\frac{A_q}{A}$, to be used in the proof of Theorem~\ref{main_thm3}.

\begin{proposition}\label{Aq_bound2}
For all integers $A,q$ with $1 \leq A \leq q$,  
\begin{equation*}
    A_q = A \cdot \frac{\varphi(q)}{q} + E, 
\end{equation*}
where $\abs{E} < 2^{\omega(q) - 1}$.
\end{proposition}

\begin{proof}
By the Principle of Inclusion-Exclusion, we have the following count:
\begin{equation*}
    A_q = \sum\limits_{d \mid q} \floor*{\frac{A}{d}} \mu(d) = \sum\limits_{d \mid q} \frac{A}{d}\mu(d) + E = A \sum\limits_{d \mid q} \frac{\mu(d)}{d} + E = A\cdot \frac{\varphi(q)}{q} + E,
\end{equation*}
where 
\begin{align*}
    E &= \sum_{d \mid q} \left(\floor*{\frac{A}{d}} - \frac{A}{d} \right) \mu(d)  \\
    &= \sum_{\substack{d \mid q \\ \mu(d) > 0 }} \left(\floor*{\frac{A}{d}} - \frac{A}{d} \right)  \mu(d) + \sum_{\substack{d \mid q \\ \mu(d) < 0 }} \left(\floor*{\frac{A}{d}} - \frac{A}{d} \right) \mu(d) \\
    &< \sum_{\substack{k \leq \omega(q) \\ k \text{ odd }}} \binom{\omega(q)}{k} = 2^{\omega(q) - 1},
\end{align*}
and similarly $E > -2^{\omega(q)-1}$.
\end{proof}
\begin{remark}\label{prop_helper_remark}
Proposition~\ref{Aq_bound2} is nontrivial if the following is satisfied for some constant $C<1$:
\begin{equation}\label{nontrivial_range}
    2^{\omega(q) - 1} \leq C \cdot \frac{\varphi(q)}{q}A.
\end{equation}
Rearranging the above, we receive the following condition on $A$:
\begin{equation*}
    A \geq \frac{1}{2C} \cdot \frac{q}{\varphi(q)} \cdot 2^{\omega(q)}.
\end{equation*}
Since later we shall impose $A = \lfloor \frac{1}{10}N q^{-\frac{1}{4}} \rfloor \geq \frac{1}{10}Nq^{-\frac{1}{4}}-1$ and $q^{\frac{3}{8}} \leq N \leq q^{\frac{5}{8}}$, we can convert the inequality above into a sufficient condition purely in terms of $q$ (subject to the aforementioned restrictions):
\begin{equation}\label{q_condition1}
    q^{\frac{1}{8}} \geq 10 \left(2^{\omega(q)}\cdot \frac{q}{\varphi(q)} \cdot \frac{1}{2C} + 1 \right),
\end{equation}
In other words, if $q$ satisfies~\eqref{q_condition1}, the condition in~\eqref{nontrivial_range} will hold for all $N$ in the relevant range. Moreover, it is clear that since $2^{\omega(q)} \leq d(q)$ is $O(q^{\varepsilon})$, and $\frac{q}{\varphi(q)} = O(\log \log q) = O(q^{\varepsilon})$,~\eqref{q_condition1} is satisfied for all sufficiently large $q$.

In particular, for sufficiently large $q$, we have the following lower bound for $\frac{A_q}{A}$ coming from~\eqref{nontrivial_range}:
\begin{equation}\label{prop_helper_eq}
    \frac{A_q}{A} \geq (1-C)\frac{\varphi(q)}{q} > \frac{1-C}{e^{\gamma}\log\log q + \frac{3}{\log\log q}},
\end{equation}
which is on the order of $\frac{1}{\log \log q}$.
\end{remark}
In the following proposition, we determine an explicit sufficient condition on the size of $q$ for~\eqref{prop_helper_eq} to hold with $C = \frac{1}{2}$:
\begin{proposition}\label{Aq_bound3}
For positive integers $q,A,N$ satisfying $q \geq e^{e^{9.594}}$, $A = \lfloor \frac{1}{10}N q^{-\frac{1}{4}} \rfloor \geq \frac{1}{10}Nq^{-\frac{1}{4}}-1$, and $q^{\frac{3}{8}} \leq N \leq q^{\frac{5}{8}}$, we have
\begin{equation}
    \frac{A_q}{A} \geq \frac{1}{2} \cdot \frac{\varphi(q)}{q} > \frac{1}{2} \cdot \frac{1}{e^{\gamma}\log\log q + \frac{3}{\log\log q}}.
\end{equation}
\end{proposition}
\begin{proof}
By Remark~\ref{prop_helper_remark}, it is enough to determine a sufficient condition on the size of $q$ for~\eqref{q_condition1} to hold. We can use the bounds $\varphi(q) > \frac{q}{e^{\gamma}\log\log q + \frac{3}{\log\log q}}$ and $2^{\omega(q)} \leq d(q) \leq q^{\frac{1.066}{\log\log q}}$ (Theorems~\ref{phi_bound} and~\ref{rs_divisor_bound}, respectively), to obtain the following condition on $q$:
\begin{equation}\label{q_condition2}
    q^{\frac18} \geq 10 \left( \frac{1}{2C} q^{\frac{1.066}{\log\log q}} \left( e^{\gamma}\log\log q + \frac{3}{\log\log q} \right) + 1 \right).
\end{equation}

Hence, for $C = \frac{1}{2}$ and $q \geq e^{e^{9.594}}\approx 8.03104 \cdot 10^{6373}$, it is sufficient that
\begin{equation}\label{q_condition3}
    q^{\frac{1}{8}} \geq 10 \left( q^{\frac{1}{9}} \left( e^{\gamma}\log\log q + \frac{1}{3\cdot 1.066} \right) + 1 \right).
\end{equation}
It is straightforward that~\eqref{q_condition3} holds for all $q \geq e^{e^{9.594}}$. It follows that~\eqref{q_condition1} (with $C = \frac{1}{2}$) holds for all $q$ at least $e^{e^{9.594}}$ as well.
\end{proof}

\section{Fourth-Moment Bound}\label{Lemma8}
In this section, we will prove an explicit version of~\cite[Lemma 8]{burgess1}. This lemma is analogous to~\cite[Lemma 9.26]{mv}, and in our generalization of the proof presented in Montgomery and Vaughan to composite moduli, Lemma~\ref{lemma:fourth_moment_of_character_sums} will take the place of that result.

For the proof, we will utilize the following bound, which is the $r=2$ case of~\cite[Lemma 7]{burgess1}:
\begin{lemma}[\text{\cite[Lemma~7]{burgess1}}]\label{lemma7burgess}
Let $\chi$ be a primitive character$\mod q$, and
$$
f_{1}(x)=\left(x-m_{1}\right)\left(x-m_{2}\right)\quad  f_{2}(x)=\left(x-m_{3}\right)\left(x-m_{4}\right)$$
$$f(x)=f_{1}(x) f_{2}(x)^{\varphi(q)-1}
$$
where at least three of $m_1, ... ,m_{4}$ are distinct. Define
$$A_{i}=\prod_{j \neq i}\left(m_{i}-m_{j}\right)$$
Then,
$$
\left|\sum_{x=1}^{q} \chi(f(x))\right| \leqslant 8^{\omega(q)} q^{\frac{1}{2}}\left(q, A_{i}\right)
$$
for some $A_i\neq 0$.
\end{lemma}
Just as the proof of~\cite[Lemma 9.26]{mv} relies on Weil's proof of the Riemann Hypothesis for curves over a finite field, Weil's result is used in the proof of Burgess' lemma above. 

The following is the main result of this section:
\begin{lemma}\label{lemma:fourth_moment_of_character_sums}
Let $\chi$ be any primitive character modulo $q$, $B<q^\frac12$ an integer. Then,
\begin{equation*}
\sum\limits_{l=1}^{q}\abs{\sum\limits_{b=1}^{B}\chi(l+b)}^4 \leq (7B^2-6B)q + 4\cdot 8^{\omega(q)}q^{\frac{1}{2}}B^4\left(d(q)\right)^3.
\end{equation*}
\end{lemma}

\begin{proof}
Expanding and applying the triangle inequality, we have 
$$
\sum_{l=1}^{q}\abs{\sum_{b=1}^B \chi(l+b)}^4\leq \sum_{b_1=1}^{B}\sum_{b_2=1}^{B}\sum_{b_3=1}^{B}\sum_{b_4=1}^{B}\abs{\sum_{l=1}^q\chi((l+b_1)(l+b_2))\bar{\chi}((l+b_3)(l+b_4))}.$$
In the summation over $b_i$'s, the contribution to the sum from terms consisting of at most 2 distinct $b_i$'s is at most
$$\left(B\cdot 1+{{B}\choose{2}}\cdot(2^4-2)\right) \abs{\sum_{l=1}^{q}\chi((l+b_1)(l+b_2))\bar{\chi}((l+b_3)(l+b_4))}\leq (7B^2-6B)q.$$
The remaining tuples in the summation consist at least 3 distinct $b_i$'s. For the tuples in this set, we may apply Lemma~\ref{lemma7burgess}:
\begin{equation*}
    \abs{\sum\limits_{x=1}^{q} \chi(l+b_1)\chi(l+b_2)\overline{\chi}(l+b_3)\overline{\chi}(l+b_4)} \leq 8^{\omega(q)} q^{\frac{1}{2}} (q,A_i).
\end{equation*}
Hence,
\begin{equation}\label{second_case_bound}
    \begin{split}
        \sum'\limits_{1\leq b_1,\dots,b_4\leq B}  \abs{\sum\limits_{l=1}^{q} \chi(l+b_1)\chi(l+b_2)\overline{\chi}(l+b_3)\overline{\chi}(l+b_4)}
        \leq 8^{\omega(q)}q^{\frac{1}{2}}\sum'\limits_{1\leq b_1,\dots,b_4\leq B} \sum\limits_{\substack{j=1 \\ A_j \neq 0}}^{4} (q,A_j).
    \end{split}
\end{equation}
Where $\sum^{'}$ denotes a sum over $b_i$'s with at least $3$ of them distinct. As $B<q^\frac12$,
\begin{align*}
    \sum'\limits_{1\leq b_1,\dots,b_4\leq B} \sum\limits_{\substack{j=1 \\ A_j \neq 0}}^{4} (q,A_j) &\leq \sum\limits_{j=1}^{4}8B\sum_{\substack{|b_i-b_j|\in [B-1] \\ i\neq j}}^{\prime}\prod_{ i\neq j}\left((b_i-b_j),q\right)\\
    &\leq 32B \left(\sum_{k=1}^{B-1}(k,q)\right)^3\\
    &\leq 32B \left(\sum_{\substack{d\mid q\\ d\leq B-1}}\floor*{\frac{B-1}{d}}d\right)^3\\
    &\leq 32B(B-1)^3\left(\frac{1}{2}d(q)\right)^3\\
    &= 4B(B-1)^3\left(d(q)\right)^3,
\end{align*}
where the $8B = 2^3 B$ factor in the first line results from the observation that $\abs{b_i-b_j} = k$ gives at most two choices for $b_i$ for each (fixed) choice of $b_j$. 

Utilizing the above in~\eqref{second_case_bound}, we receive
\begin{align*}
    \sum'\limits_{1\leq b_1,\dots,b_4\leq B}  \abs{\sum\limits_{x=1}^{q} \chi(l+b_1)\chi(l+b_2)\overline{\chi}(l+b_3)\overline{\chi}(l+b_4)}
    &\leq 8^{\omega(q)}q^{\frac{1}{2}}\sum'\limits_{1\leq b_1,\dots,b_4\leq B} \sum\limits_{\substack{j=1 \\ A_j \neq 0}}^{4} (q,A_j) \\
    &\leq 4\cdot 8^{\omega(q)}q^{\frac{1}{2}}B^4\left(d(q)\right)^3
\end{align*}
as a bound for the contribution to our original sum from tuples which contain more than $2$ distinct elements. Combining this with our bound of $(7B^2-6B)q$ in the other case completes the proof.
\end{proof}
\section{Proof of Theorem~\ref{main_thm3}}\label{1.2_proof}
In the following lemma, we prove that under certain assumptions on the size of $q$ and $N$, the conclusion of Theorem~\ref{main_thm3} holds as a consequence of the trivial bound.
\begin{lemma}\label{technical_lemma2}
Let $q \geq e^{e^{9.594}}$, and
\begin{equation}\label{technical_lemma2_eq}
    \lambda_2' = 3.3325  \log^{\frac{1}{4}}(q) \left( \frac{q}{\varphi(q)} \right)^{\frac{1}{2}} (2^{\omega(q)}d(q))^{\frac{3}{4}} .
\end{equation}
Let $\chi$ be a primitive character of modulus $q$. Then, if either of the below two conditions hold
\begin{enumerate}[label=(\alph*)]
    \item $N \leq q^{\frac38}$
    \item $N \leq  e^{\frac{3}{8} \cdot e^{9.594}}$,
\end{enumerate}
then
\begin{equation*}
    \abs{S(M,N)} \leq \lambda_2' \sqrt{N} q^{\frac{3}{16}}.
\end{equation*}
\end{lemma}
\begin{proof}
If $N\leq \lambda_2'^2 q^{\frac38}$, the trivial bound yields $$\abs{S(M,N)}\leq N \leq N^{\frac{1}{2}}q^{\frac{3}{16}} \lambda_2',$$ proving (a), since $\lambda_2' \geq 1$. Claim (b) follows from (a) combined with the assumption on the size of $q$.
\end{proof}
We now proceed with the proof of Theorem~\ref{main_thm3}.  By Lemma~\ref{technical_lemma2}, we may assume $q^\frac38 \leq N \leq q^\frac58$. Define
$$\mathcal{M}(y)=\max _{\substack{M,N \\ N\leq y}}|S(M, N)|.$$
Then
$$S(M, N)=\sum_{n=M+1}^{M+N} \chi(n+a b)+2 \theta \mathcal{M}(a b),$$
where $|\theta| \leq 1$. Take $A=\floor*{\frac{1}{10} N q^{-\frac{1}{4}}}$, $B=\floor*{q^{\frac{1}{4}}}$. Clearly, $B\geq 1$. Now, it is vital to our argument that $A \geq 1$ as well. However, the lower bound on $A$ is guaranteed by the assumption $q \geq e^{e^{9.594}}$ in the hypothesis of Theorem~\ref{main_thm3}; in fact, we have $A\geq \frac{1}{10}q^{\frac{1}{8}}>5\cdot 10^{795}$.

Recall that $A^q$ denotes the set of integers from $1$ to $A$ that are relatively prime to $q$. As $A, B\geq 1$, summing over $a\in A^q$ and $b \in [1,B]$, we find that
\begin{equation}\label{recursion_linchpin}
A_q B S(M, N)=\sum_{n=M+1}^{M+N}\sum_{a\in A_q} \sum_{b=1}^{B}\chi(n+a b)+2 A_q B \theta_{1} \mathcal{M}(A B).
\end{equation}
Now notice
\begin{equation*}
\abs{\sum_{n=M+1}^{M+N}\sum_{a\in A_q} \sum_{b=1}^{B}\chi(n+a b)}=\sum_{l=1}^{q}v(l)\abs{\sum_{b=1}^B \chi(l+b)},
\end{equation*}
where $v(l)$ is the number of pairs $(a,n)$ such that $a\in A^q, n\in [M+1,M+N], n\equiv al \Mod q$.
By H\"{o}lder's Inequality,
\begin{equation}\label{holder1}
\abs{\sum_{n,a,b}\chi(n+a b)}^4\leq\left(\sum_{l=1}^q v(l)\right)^2 \left(\sum_{l=1}^q v(l)^2\right)\left(\sum_{l=1}^{q}\abs{\sum_{b=1}^B \chi(l+b)}^4\right),
\end{equation}
and clearly
\begin{equation*}
\left(\sum_{l=1}^q v(l)\right)^2=(A_qN)^2=A_q^2N^2.
\end{equation*}
From Lemma~\ref{lemma:fourth_moment_of_character_sums}, we have
\begin{equation}\label{from_previous_lemma1}
\sum_{l=1}^{q}\abs{\sum_{b=1}^B \chi(l+b)}^4 \leq (7B^2-6B)q + 4\cdot 8^{\omega(q)}q^{\frac{1}{2}}B^4\left(d(q)\right)^3.
\end{equation}

Now we bound $\displaystyle\sum_{l=1}^q v(l)^2$.
This is the number of choices $a,a'\in A^q$, $n,n'\in [1,N]$, such that
$$\begin{cases} M+n\equiv al\Mod q \\ M+n'\equiv a'l\Mod q \end{cases}
$$
By elimination of $l$, it suffices to count the number of tuples $(a,n,a',n')$ with $a,a' \in A^q$ and $n,n' \in [1,N]$ satisfying 
\begin{equation}\label{counting_condition}
(a'-a)M \equiv an'-a'n \Mod q. 
\end{equation}

For any tuple $(a,a',n,n')$ satisfying~\eqref{counting_condition}, pick $\abs{k} < \frac{q}{2}$ with $k \equiv (a-a')M \Mod q$. By our assumption $N \leq q^{\frac{5}{8}}$, it follows that
$1\leq an', a'n\leq AN\leq \frac{1}{10}N^2q^{-\frac{1}{4}}\leq \frac{q}{10}$. Thus, $\abs{an'-a'n}\leq \ceil{\frac{q}{10}} < \frac{q}{2}$, and $a'n - an' = k$.

Hence, given $(a,a',n_0,n_0')$ satisfying~\eqref{counting_condition}, the other pairs $(n,n')$ satisfying~\eqref{counting_condition} \textit{for the same $(a,a')$ pair} are of the form $n = n_0 + \frac{a}{(a,a')}h$, $n' = n_0'+\frac{a'}{(a,a')}h$, with $\abs{h} \leq \frac{N(a,a')}{\max\{a,a'\}}$. Consequently, there are at most $1+\frac{N(a,a')}{\max\{a,a'\}}$ such pairs. Thus,
\begin{align*}\label{sum} 
\sum_{l} v(l)^{2} &\leq\sum_{a,a'\in A^q}\left(1+\frac{N(a, a^{\prime})}{\max \left\{a, a^{\prime}\right\}}\right) \\
&\leq A_q^2+2\sum_{\substack{a\leq a'\\a,a' \in A^q}}\frac{N(a, a^{\prime})}{a'} \nonumber \\
&= A_q^2+2\sum_{d\in A^q}\sum_{\substack{1\leq b\leq b^{\prime}\leq A/d\\(bb^{\prime},q)=(b,b^{\prime})=1 }}\frac{N}{b'}  \\
&\leq A_q^2 + \sum_{d \in A^{q}} \sum_{\substack{b' \leq A/d \\ (b', q)=1}} b^{\prime} \frac{N}{b^{\prime}} \\
&\leq A_q^2+2\sum_{d\in A^q} N \frac{A}{d}\\
&= A_q^2 + 2AN \sum\limits_{d \in A^q}\frac{1}{d}.
\end{align*}

To bound $\sum_{d \in A^q} \frac{1}{d}$ from above, we can simply ``push" the terms of the sum up so that the denominators are in the set $\{1,2,\dots,A_q\}$. In other words, $\sum_{d \in A^q} \frac{1}{d} \leq \sum\limits_{1 \leq d \leq A_q} \frac{1}{d} \leq \log (2A_q)$, giving
\begin{equation}\label{vl_bound2}
    \sum_{l} v(l)^{2} \leq A_q^2 + 2AN \log (2A_q).
\end{equation}

Thus, applying~\eqref{from_previous_lemma1} and~\eqref{vl_bound2} in~\eqref{holder1}, we receive
$$
\begin{aligned}
\abs{\sum_{n,a,b}\chi(n+a b)}^4&\leq\left(\sum_{l=1}^q v(l)\right)^2 \left(\sum_{l=1}^q v(l)^2\right)\left(\sum_{l=1}^{q}\abs{\sum_{b=1}^B \chi(l+b)}^4\right)\\
&\leq (A_q^2N^2)( A_q^2 + 2AN \log (2A_q))\left((7B^2-6B)q + 4\cdot 8^{\omega(q)}q^{\frac{1}{2}}B^4\left(d(q)\right)^3 \right)\\
&\leq B^4 N^2 q^{\frac{1}{2}} A_q^2 \left(\frac{7q^{\frac{1}{2}}}{B^2}+4 \cdot 8^{\omega(q)}d(q)^3\right)\left(A_q^2 + 2AN \log (2A_q)\right),
\end{aligned}
$$

and therefore
$$
\begin{aligned}
\abs{\sum_{n,a,b}\chi(n+a b)}&\leq  B N^{\frac{1}{2}}q^{\frac{1}{8}} A_q^{\frac{1}{2}} \left(7\left(\frac{q^{\frac{1}{4}}}{q^{\frac{1}{4}}-1}\right)^2+4 \cdot 8^{\omega(q)}d(q)^3\right)^{\frac{1}{4}} \left(A_q^2 + 2AN \log (2A_q)\right)^{\frac{1}{4}},
\end{aligned}
$$
$$
\begin{aligned}
\frac{1}{A_qB}\abs{\sum_{n,a,b}\chi(n+a b)}
&\leq N^\frac{1}{2}q^{\frac{1}{8}}\left(7\left(\frac{q^{\frac{1}{4}}}{q^{\frac{1}{4}}-1}\right)^2+4 \cdot 8^{\omega(q)}d(q)^3\right)^{\frac{1}{4}} \left(1+2\frac{A}{A_q^2}N\log (2A_q)\right)^{\frac{1}{4}} \\
&\leq N^\frac{1}{2}q^{\frac{1}{8}}\left(7\left(\frac{q^{\frac{1}{4}}}{q^{\frac{1}{4}}-1}\right)^2+4 \cdot 8^{\omega(q)}d(q)^3\right)^{\frac{1}{4}} \left(1 + 20\frac{A(A+1)}{A_q^2} q^{\frac{1}{4}} \cdot \left(\frac{3}{8} \log (q) - \log(5) \right)\right)^{\frac{1}{4}}  \\
&\leq N^\frac{1}{2}q^{\frac{1}{8}}\left(7\left(\frac{q^{\frac{1}{4}}}{q^{\frac{1}{4}}-1}\right)^2+4 \cdot 8^{\omega(q)}d(q)^3\right)^{\frac{1}{4}} \left(7.5\left(\frac{A}{A_q}\right)^2 q^{\frac{1}{4}} \log (q)\right)^{\frac{1}{4}} \\
&= N^\frac{1}{2}q^{\frac{3}{16}} \log^{\frac{1}{4}}(q) \left(7\left(\frac{q^{\frac{1}{4}}}{q^{\frac{1}{4}}-1}\right)^2+4 \cdot 8^{\omega(q)}d(q)^3\right)^{\frac{1}{4}} \left(7.5\left(\frac{A}{A_q}\right)^2\right)^{\frac{1}{4}}, 
\end{aligned}
$$
where we have used that $A \geq \frac{1}{10}N q^{-\frac{1}{4}} - 1$ implies $N \leq 10(A+1)q^{\frac{1}{4}}$, and $A_q \leq A \leq \frac{1}{10} N q^{-\frac{1}{4}} \leq \frac{1}{10} q^{\frac{3}{8}}$ implies $\log (2 A_q) \leq \frac{3}{8}\log(q) - \log(5)$ for the second inequality. We also used $1 - \frac{\log(5)}{2} \frac{20 A(A+1)}{A_q^2}q^{\frac{1}{4}} \leq 0 $ and $(A+1)\left(\frac{3}{8}\log(q) - \frac{\log 5}{2} \right) \leq A \cdot \frac{3}{8} \log (q)$ for the third inequality, which hold since $A \geq A_q$, $A \geq \frac{1}{10}q^{\frac{1}{8}} - 1$, and $q \geq e^{e^{9.594}}$. 

We now utilize Proposition~\ref{Aq_bound3} to bound $\frac{A_q}{A}$:
$$
\begin{aligned}
\frac{1}{A_qB}\abs{\sum_{n,a,b}\chi(n+a b)}&\leq N^\frac{1}{2}q^{\frac{3}{16}} \log^{\frac{1}{4}}(q) \left(7\left(\frac{q^{\frac{1}{4}}}{q^{\frac{1}{4}}-1}\right)^2+4 \cdot 8^{\omega(q)}d(q)^3\right)^{\frac{1}{4}} \left(7.5\left(\frac{A}{A_q}\right)^2\right)^{\frac{1}{4}} \\
&\leq N^\frac{1}{2}q^{\frac{3}{16}} \log^{\frac{1}{4}}(q)  \left(7\left(\frac{q^{\frac{1}{4}}}{q^{\frac{1}{4}}-1}\right)^2+4 \cdot 8^{\omega(q)}d(q)^3\right)^{\frac{1}{4}} \left(30\left(\frac{q}{\varphi(q)}\right)^2\right)^{\frac{1}{4}} \\
&\leq N^{\frac{1}{2}} q^{\frac{3}{16}} \log^{\frac{1}{4}}(q) \cdot 1.0001 \cdot \left(4 + \frac{7}{64} \right)^{\frac{1}{4}} (2^{\omega(q)}d(q) )^{\frac{3}{4}} \left(30\left(\frac{q}{\varphi(q)}\right)^2\right)^{\frac{1}{4}}   \\
&\leq 3.3325 N^{\frac{1}{2}} q^{\frac{3}{16}} \log^{\frac{1}{4}}(q) \left( \frac{q}{\varphi(q)} \right)^{\frac{1}{2}} (2^{\omega(q)}d(q))^{\frac{3}{4}},
\end{aligned}
$$
where above we have used 
\begin{equation*}
\left(4+7 \frac{\left(\frac{q^{1 / 4}}{q^{1 / 4}-1}\right)^{2}}{8^{\omega(q)} d(q)^{3}}\right)^{1 / 4} \leq 1.0001\left(4+\frac{7}{64}\right)^{1 / 4}
\end{equation*}
which holds because $q \geq e^{e^{9.594}}$.

Thus, by~\eqref{recursion_linchpin},
\begin{equation}\label{recursion_step2}
    \mathcal{M}(N) \leq N^{\frac{1}{2}}q^{\frac{3}{16}} \lambda_2' +2 \mathcal{M}\left(\frac{N}{10}\right),
\end{equation}
where $\lambda'_2 = 3.3325  \log^{\frac{1}{4}}(q) \left( \frac{q}{\varphi(q)} \right)^{\frac{1}{2}} (2^{\omega(q)}d(q))^{\frac{3}{4}}$. Note in particular that Lemma~\ref{technical_lemma2} shows that~\eqref{recursion_step2} holds if $N$ is smaller than $q^{3/8}$ or $e^{\frac{3}{8} \cdot e^{9.594}}$. 

Thus, applying~\eqref{recursion_step2} repeatedly, we obtain
\begin{equation*}
    \mathcal{M}(N) \leq \sqrt{N}q^{\frac{3}{16}}\lambda_2'\sum\limits_{k=0}^{K}2^{k}10^{-\frac{k}{2}} + 2^{K+1}\mathcal{M}(10^{-(K+1)}N).
\end{equation*}
Note that in the above equation,
\begin{align*}
    \sum\limits_{k=0}^{K} 2^k10^{-\frac{k}{2}} \leq \sum\limits_{k=0}^{\infty}2^k10^{-\frac{k}{2}} = \sum\limits_{k=0}^{\infty}\left(\frac{2}{\sqrt{10}}\right)^k = \frac{1}{1-\frac{2}{\sqrt{10}}} = \frac{\sqrt{10}}{\sqrt{10}-2}
\end{align*}
for any $K$, and the trivial bound gives
\begin{align*}
2^{K+1}\mathcal{M}(10^{-(K+1)}N) \leq 2^{K+1}\cdot10^{-(K+1)}N = \frac{N}{5^{K+1}}.
\end{align*}
Hence, for any $K$,
\begin{align*}
    \mathcal{M}(N) \leq \frac{\sqrt{10}}{\sqrt{10}-2}\sqrt{N}q^{\frac{3}{16}}\lambda_2' + \frac{N}{5^{K+1}}.
\end{align*}
Since the above is true for any $K \geq 0$ it follows that
\begin{equation*}
    \mathcal{M}(N) \leq \frac{\sqrt{10}}{\sqrt{10}-2}\lambda_2'\sqrt{N}q^{\frac{3}{16}},
\end{equation*}
which completes the proof of Theorem~\ref{main_thm3}. 
\appendix
\section{Useful Bounds for Arithmetic Functions}
In Sections~\ref{intro} and~\ref{Bounding_Aq_A}, we use the following bound for $\varphi(n)$, which is~\cite[Theorem 8.8.7]{outlaw}:
\begin{theorem}\label{phi_bound}
For $n \geq 3$,
\begin{equation*}
   \varphi(n)>\frac{n}{e^{\gamma} \log \log n+\frac{3}{\log \log n}}
\end{equation*}
\end{theorem}
We also provide bounds for $d(n)$ and $\omega(n)$, from Nicolas and Robin~\cite{nicorobin} and Robin~\cite[Theorem 12]{robin1} respectively:
\begin{theorem}\label{rs_divisor_bound}
    For any $n \geq 3$, 
    \begin{equation*}
        d(n) \leq n^{\frac{1.066}{\log \log n}}.
    \end{equation*}
\end{theorem}
\begin{theorem}
    For all $n \geq 3$,
    \begin{equation*}
        \omega(n) \leq \frac{\log n}{\log \log n}+1.45743 \frac{\log n}{(\log \log n)^{2}}.
    \end{equation*}
\end{theorem}

\end{document}